\documentclass[10pt]{amsart}
\usepackage{amsmath,amsfonts,amssymb,amsthm}
\usepackage[alphabetic]{amsrefs}
\usepackage{url}
\usepackage{graphicx,color}

\newtheorem{theorem}{Theorem}[section]
\newtheorem{lemma}[theorem]{Lemma}

\newtheorem*{theorem*}{Theorem}

\theoremstyle{definition}
\newtheorem{definition}{Definition}[section]

\theoremstyle{remark}

\usepackage{caption}
\begin{document}

\title[\ ]{}

\begin{center}
\uppercase{\textbf{Quadratic isoperimetric inequality for $7$-located simplicial complexes}}\\
\vspace{0.5cm}
\end{center}

\author[\ ]{
Ioana-Claudia Laz\u{a}r\\
Politehnica University of Timi\c{s}oara, Dept. of Mathematics,\\
Victoriei Square $2$, $300006$-Timi\c{s}oara, Romania\\
E-mail address: ioana.lazar@upt.ro}

\date{}

\begin{abstract}

We show that $7$-located simplicial complexes satisfy a quadratic isoperimetric inequality.

\hspace{0 mm} \textbf{2010 Mathematics Subject Classification}:
05C99, 05C75.

\hspace{0 mm} \textbf{Keywords}: $7$-location, quadratic isoperimetric inequality, minimal disc diagram.
\end{abstract}

\pagestyle{myheadings}

\markboth{}{}

  \vspace{-10pt}

\maketitle

\section{Introduction}

Curvature can be expressed both in metric and combinatorial terms. Metrically, one can refer to ’nonpositively
curved’ (respectively, ’negatively curved’) metric spaces in the sense of Aleksandrov, i.e. by comparing small triangles in the space with
triangles in the Euclidean plane (hyperbolic plane). These are the CAT(0) (respectively, CAT(-1)) spaces.

Combinatorially, one looks for local combinatorial conditions implying
some global features typical for nonpositively curved metric spaces.
A very important combinatorial condition of this type was formulated by Gromov \cite{Gro} for cubical complexes, i.e.\ cellular complexes
with cells being cubes. Namely, simply connected cubical complexes with links (that can be thought as small spheres around vertices)
being flag (respectively, $5$-large, i.e.\ flag-no-square) simplicial complexes carry a canonical CAT(0) (respectively, CAT(-1)) metric.
Another important local combinatorial condition is local $k$--largeness, introduced independently by Chepoi \cite{Ch} (under the name of bridged complexes), Januszkiewicz-{\' S}wi{\c a}tkowski \cite{JS1} and Haglund \cite{Hag}. A flag simplicial complex is \emph{locally $k$-large} if its links do not contain `essential' loops of length less than $k$.

In \cites{O-sdn, ChOs,BCCGO,ChaCHO} some other curvature conditions
are studied -- they form a way of unifying CAT(0) cubical and systolic theories.
On the other hand, Osajda \cite{O-8loc} introduced a local combinatorial condition called \emph{$m$-location}, and used it, for $m = 8$, to provide a new solution to Thurston's
problem about hyperbolicity of some $3$-manifolds. In \cite{L-8loc} and \cite{L-8loc2} a systematic study of a version of $m$-location, suggested in \cite{O-8loc}, is undertaken. This version is in a sense more natural than the original one (tailored to Thurston's problem), and neither of them is implied by the other.
 Roughly, the new $m$-location says that essential loops of length at most $m$ admit filling diagrams with at most one internal vertex. In \cite{L-8loc} (Theorem $4.3$) it is shown that $8$-location is a negative-curvature-type condition. Namely, it is proven that simply connected,
$8$-located simplicial complexes are Gromov hyperbolic. In \cite{L-8loc2} we introduce another combinatorial curvature condition, called the $5/9$-condition, and we show that the complexes which fulfill it, are also Gromov hyperbolic.

Isoperimetric inequalities relate the length of closed curves to the infimal area of
the discs which they bound. It is well-known that every closed loop of length $L$ in
the Euclidean plane bounds a disc whose area is less than $\frac{L^{2}}{4 \pi}$, and this bound is
optimal. Thus one has a quadratic isoperimetric inequality for loops in Euclidean
space. In contrast, loops in real hyperbolic space satisfy a linear isoperimetric inequality:
there is a constant $C$ such that every closed loop of length $L$ in hyperbolic space bounds a
disc whose area is less than or equal to $C \cdot L$. It is known
that (with a suitable notion of area) a geodesic space $X$ is $\delta$-hyperbolic if and only
if loops in $X$ satisfy a linear isoperimetric inequality (see \cite{BH}, chapter $III.H$, page $417$ and page $419$).
Both $8$-located complexes and $5/9$-complexes satisfy therefore, under the additional hypothesis of simply connectedness, a linear isoperimetric inequality (see \cite{L-8loc}, \cite{L-8loc2}). For loops in arbitrary CAT(0) spaces, however, there is
a quadratic isoperimetric inequality (see \cite{BH}, chapter $III.H$, page $414$).

It is known that cycles in systolic complexes satisfy a quadratic isoperimetric inequality (see \cite{JS1}).
In \cite{E1} explicit constants are provided presenting the optimal estimate
on the area of a systolic disc.
In systolic complexes the isoperimetric function for $2$-spherical cycles
(the so called second isoperimetric function) is linear (see \cite{JS2}).
In \cite{ChaCHO} it is shown that meshed graphs (thus, in particular, weakly modular graphs) satisfy a quadratic isoperimetric inequality.

The purpose of the current paper is to show that for cycles in $7$-located complexes, there is a quadratic isoperimetric inequality.
We prove that the disc in the diagram associated to a cycle in a simply connected, $7$-located complex, is itself $7$-located. Then we show that such a disc satisfies a quadratic isoperimetric inequality. To prove this, we use a method introduced in \cite{ChaCHO}.

\textbf{Acknowledgements.} The author would like to thank Damian Osajda for introducing her to the subject.
This work was partially supported by the grant $346300$ for IMPAN from the Simons Foundation and the matching $2015-2019$ Polish MNiSW fund.

\section{Preliminaries}

Let $X$ be a simplicial complex. We denote by $X^{(k)}$ the $k$-skeleton of
$X, 0 \leq k < \dim X$. A subcomplex $L$ in $X$ is called \emph{full} as a subcomplex of $X$ if any simplex of $X$ spanned by a set of vertices in $L$, is a simplex of $L$. For a set
$A = \{ v_{1}, ..., v_{k} \}$ of vertices of $X$, by $\langle  A \rangle$ or by $\langle  v_{1}, ..., v_{k} \rangle$ we denote the \emph{span} of $A$, i.e. the
smallest full subcomplex of $X$ that
contains $A$. We write $v \sim v'$ if $\langle  v,v' \rangle \in X$ (it can happen that $v = v'$). We write $v \nsim v'$ if $\langle  v,v' \rangle \notin X$.
 We call $X$ {\it flag} if any finite set of vertices which are pairwise connected by
edges of $X$, spans a simplex of $X$.

A {\it cycle} ({\it loop}) $\gamma$ in $X$ is a subcomplex of $X$ isomorphic to a triangulation of $S^{1}$. A \emph{full cycle} in $X$ is a cycle that is full as a subcomplex of $X$.
A $k$-\emph{wheel} in $X$ $(v_{0}; v_{1}, ..., v_{k})$ (where $v_{i}, i \in \{0,..., k\}$
are vertices of $X$) is a subcomplex of $X$ such that $(v_{1}, ..., v_{k})$ is a full cycle and $v_{0} \sim v_{1}, ..., v_{k}$.
The \emph{length} of $\gamma$ (denoted by $|\gamma|$) is the number of edges in $\gamma$.

We define the \emph{metric} on the $0$-skeleton of $X$ as the number of edges in the shortest $1$-skeleton path joining two given vertices and we denote it by $d$. A \emph{ball (sphere)}
$B_{i}(v,X)$ ($S_{i}(v,X)$) of radius $i$ around some vertex $v$ is a full subcomplex of $X$ spanned by vertices at distance at most $i$ (at distance $i$) from $v$.

\begin{definition}\label{def-2.1}
A simplicial complex is $m$-\emph{located} if it is flag and every full homotopically trivial loop of length at most $m$ is contained in a $1$-ball.
\end{definition}

Let $\sigma$ be a simplex of $X$. The \emph{link} of $X$ at $\sigma$, denoted $X_{\sigma}$, is the subcomplex of $X$ consisting of all simplices of $X$ which are disjoint from $\sigma$ and which, together
with $\sigma$, span a simplex of $X$.

\begin{definition}\label{2.2}
A \emph{simplicial map} $f : X \rightarrow Y$ between simplicial complexes $X$ and $Y$ is a map
which sends vertices to vertices, and whenever vertices $v_{0}, ..., v_{k} \in X$ span a simplex
$\sigma$ of $X$ then their images span a simplex $\tau$ of $Y$ and we have $f(\sigma) = \tau$.
Therefore a simplicial map is determined by its values on the vertex set of $X$. A
simplicial map is called \emph{nondegenerate} if it is injective on each simplex.
\end{definition}

\begin{definition}\label{2.3}
Let $\gamma$ be a cycle in $X$. A \emph{filling diagram} for $\gamma$ is a simplicial map $f : D \rightarrow X$ where $D$ is a triangulated $2$-disc, and
$f | _{\partial D}$ maps $\partial D$ isomorphically onto $\gamma$. We denote a filling diagram for $\gamma$ by $(D,f)$ and we say it is:

$\bullet$ \emph{minimal} if $D$ has minimal area i.e. it consists of the least possible number of $2$-simplices among filling diagrams for $\gamma$;

$\bullet$ \emph{nondegenerate} if $f$ is a nondegenerate map;

\end{definition}

\begin{lemma}\label{2.4}
Let $X$ be a simplicial complex and let $\gamma$ be a homotopically trivial loop in $X$. Then:
\begin{enumerate}
\item there exists a filling diagram $(D, f)$ for $\gamma$ (see \cite{Ch} - Lemma $5.1$, \cite{JS1} - Lemma $1.6$ and \cite{Pr} - Theorem $2.7$);
\item any minimal filling diagram for $\gamma$ is simplicial and nondegenerate (see \cite{Ch} - Lemma $5.1$, \cite{JS1} - Lemma $1.6$, Lemma $1.7$ and \cite{Pr} - Theorem $2.7$).

\end{enumerate}
\end{lemma}

\begin{lemma}\label{2.5}
Let $X$ be a simplicial complex and let $\gamma$ be a homotopically trivial loop in $X$. Let $(D,f)$ be a minimal filling diagram for $\gamma$. Then adjacent $2$-simplices of $D$ have distinct images under $f$ (see \cite{Ch} - Lemma $5.1$).

\end{lemma}

Let $D$ be a simplicial disc.
We denote by $C$ the cycle bounding $D$ and by $\rm{Area} C$ the area of $D$. We denote by $V_{i}$ and $V_{b}$ the numbers of internal and boundary
vertices of $D$, respectively. Then:
$\rm{Area} C = 2 V_{i} + V_{b} - 2 = |C| + 2 (V_{i} - 1)$ (Pick's formula).
In particular, the area of a simplicial disc depends only on the numbers of its internal and
boundary vertices.

\begin{definition}\label{2.6}
Given a path $\gamma = (v_{0}, v_{1}, ..., v_{n})$ in a simplicial complex $X$, one can \emph{tighten} it to a full path $\gamma'$ with the same endpoints by repeatedly applying the following operations:

$\bullet$ if $v_{i}$ and $v_{j}$ are adjacent in $X$ for some $j > i+1$, then remove from the sequence all $v_{k}$ where $i < k < j$;

$\bullet$ if $v_{i}$ and $v_{j}$ coincide in $X$ for some $j > i$, then remove from the sequence all $v_{k}$ where $i < k \leq j$.

The tightening of a full loop is the loop itself.

\end{definition}

\section{Quadratic isoperimetric inequality for $7$-located complexes}

We start with a useful lemma.

\begin{lemma}\label{3.1}
Let $X$ be a simplicial complex and let $\gamma$ be a homotopically trivial loop in $X$. Let $(D, f)$ be a minimal filling diagram for $\gamma$.
We consider in $D$ an interior vertex $v$ such that $D_{v} \leq k$, $4 \leq k \leq 7$. Then the map $f$ is injective on $X_{f(v)}$.
\end{lemma}

\begin{proof}

Let $D_{v} = (v_{1}, v_{2}, ..., v_{k}), 4 \leq k \leq 7$.
Because $(D,f)$ is a minimal filling diagram, Lemma \ref{2.4} implies that the map $f$ is simplicial and nondegenerate. Therefore, since in $D$ there are simplices $v, v_{j}$, $\langle v, v_{j} \rangle$, $1 \leq j \leq k$, $\langle v_{j-1}, v_{j} \rangle$, $2 \leq j \leq k$, $\langle v_{k},v_{1} \rangle$, in $X$ there are simplices $f(v), f(v_{j})$, $\langle f(v), f(v_{j}) \rangle$, $1 \leq j \leq k$, $\langle f(v_{j-1}), f(v_{j}) \rangle$, $2 \leq j \leq k$, $\langle f(v_{k}),f(v_{1}) \rangle$.
Lemma \ref{2.5} implies that adjacent $2$-simplices of $D$ have distinct images under $f$.
Hence $f(v_{i\, mod\, k\; +\; 1}) \neq f(v_{(i+2)\, mod\, k\; +\; 1}), 1 \leq i \leq k$.

We show further that $f(v_{j\, mod\, k\; +\; 1}) \neq f(v_{(j+3)\, mod\, k\; +\; 1})$, $1 \leq j \leq k$.
Suppose by contradiction there exists $i$ such that $f(v_{i\, mod\, k\; +\; 1}) = f(v_{(i+3)\, mod\, k\; +\; 1})$, $1 \leq i \leq k$.
We choose a filling diagram $(D',f')$ for $\gamma$ such that in $D'$ we have
$v_{i\, mod\, k\; +\; 1} \sim v_{(i+j)\, mod\, k\; +\; 1}$, $2 \leq j \leq 4$.
We triangulate $D'$ with the same simplices like $D$ except for the triangles
$\langle v, v_{(i+j)\, mod\, k\; +\; 1}, v_{(i+j+1)\, mod\, k\; +\; 1} \rangle, 0 \leq j \leq 3$ in $D$ which are replaced in $D'$ by the triangles
$\langle v, v_{i\, mod\, k\; +\; 1}, v_{(i+4)\, mod\, k\; +\; 1} \rangle,$ $\langle v_{i\, mod\, k\; +\; 1}, v_{(i+j)\, mod\, k\; +\; 1}, v_{(i+j+1)\, mod\, k\; +\; 1} \rangle,$ $1 \leq j \leq 3$.
We define $f'$ such that it coincides with $f$ on all simplices which are common to $D$ and $D'$.
We define $f'$ such that $f'(v_{i\, mod\, k\; +\; 1}) = f'(v_{(i+3)\, mod\, k\; +\; 1}) = f(v_{i\, mod\, k\; +\; 1})$,\\
$f'(\langle v_{i\, mod\, k\; +\; 1}, v_{(i+3)\, mod\, k\; +\; 1} \rangle) = f(v_{i\, mod\, k\; +\; 1})$.
As argued above $f(v_{j\, mod\, k\; +\; 1}) \neq f(v_{(j+2)\, mod\, k\; +\; 1}), 1 \leq j \leq k$.
Since in $X$ we have $f(v_{(i+3)\, mod\, k\; +\; 1})$ $\sim$ \\ $f(v_{(i+2)\, mod\, k\; +\; 1})$, we may define $f'$ such that $f'(v_{(i+3)\, mod\, k\; +\; 1})$ $\sim$ $f'(v_{(i+2)\, mod\, k\; +\; 1})$.
Then because $f'(v_{i\, mod\, k\; +\; 1})$ $=$ $f'(v_{(i+3)\, mod\, k\; +\; 1})$, we have $f'(v_{i\, mod\, k\; +\; 1})$ $\sim$ \\ $f'(v_{(i+2)\, mod\, k\; +\; 1})$.
We define $f'$ such that $f'(\langle v_{i\, mod\, k\; +\; 1}, v_{(i+2)\, mod\, k\; +\; 1} \rangle)$ $=$ \\ $\langle f'(v_{i\, mod\, k\; +\; 1}), f'(v_{(i+2)\, mod\, k\; +\; 1}) \rangle$ $=$ $\langle f(v_{i\, mod\, k\; +\; 1}), f(v_{(i+2)\, mod\, k\; +\; 1}) \rangle$.
One can similarly show that we can define $f'$ such that $f'(\langle v_{i\, mod\, k\; +\; 1}, v_{(i+4)\, mod\, k\; +\; 1} \rangle) = \langle f'(v_{i\, mod\, k\; +\; 1}), f'(v_{(i+4)\, mod\, k\; +\; 1}) \rangle = \langle f(v_{i\, mod\, k\; +\; 1}), f(v_{(i+4)\, mod\, k\; +\; 1}) \rangle$.
We define $f'$ such that $f'(\langle v, v_{i\, mod\, k\; +\; 1}, v_{(i+4)\, mod\, k\; +\; 1} \rangle)$ $=$ \\ $\langle f(v), f(v_{i\, mod\, k\; +\; 1}),$
$f(v_{(i+4)\, mod\, k\; +\; 1})$ $\rangle$,\\
$f'(\langle v_{i\, mod\, k\; +\; 1}, v_{(i+j)\, mod\, k\; +\; 1}, v_{(i+j+1)\, mod\, k\; +\; 1} \rangle)$ $=$ \\ $\langle f(v_{i\, mod\, k\; +\; 1}), f(v_{(i+j)\, mod\, k\; +\; 1}), f(v_{(i+j+1)\, mod\, k\; +\; 1}) \rangle, 1 \leq j \leq 3$.
Hence, since $f$ is simplicial, $f'$ is also simplicial.
So $(D', f')$ is indeed a filling diagram for $\gamma$.
Note that $D$ and $D'$ have the same area.
Therefore $D'$ has minimal area. Then Lemma \ref{2.4} implies that the map $f'$ is nondegenerate.
But since $f'(v_{i\, mod\, k\; +\; 1}) = f'(v_{(i+3)\, mod\, k\; +\; 1})$, $f'$ is degenerate.
Because we have reached a contradiction, $f(v_{i\, mod\, k\; +\; 1}) \neq f(v_{(i+3)\, mod\, k\; +\; 1}),$ $1 \leq i \leq k$.

In conclusion the map $f$ is injective on $X_{f(v)}$.

\end{proof}

Next we prove the minimal filling diagrams lemma for $7$-located simplicial complexes.

\begin{lemma}\label{3.6}
Let $X$ be a $7$-located simplicial complex and let $\gamma$ be a homotopically trivial loop in $X$. Let $(D,f)$ be a minimal filling diagram for $\gamma$. Then $D$ is $7$-located.
\end{lemma}

\begin{proof}

Because $(D,f)$ is a minimal filling diagram,  Lemma \ref{2.4} implies that the map $f$ is simplicial and nondegenerate. Therefore, since in $D$ there are simplices $v, v_{i}$, $\langle v, v_{i} \rangle$, $1 \leq i \leq k$, $\langle v_{i-1}, v_{i} \rangle$, $2 \leq i \leq k$, $\langle v_{1},v_{k} \rangle$, in $X$ there are simplices $f(v), f(v_{i})$, $\langle f(v), f(v_{i}) \rangle$, $1 \leq i \leq k$, $\langle f(v_{i-1}), f(v_{i}) \rangle$, $2 \leq i \leq k$, $\langle f(v_{1}),f(v_{k}) \rangle$.

Let $\beta = (w_{1}, ... , w_{7})$ be a full cycle in $X$. Because $X$ is $7$-located and $\beta$ has length $7$, it is contained in the link of a vertex $x$.
Since $f$ is simplicial and nondegenerate, there are vertices $v_{i} \in D, 1 \leq i \leq 7$ such that $f(\langle v_{i},v_{i+1} \rangle) = \langle w_{i},w_{i+1} \rangle, 1 \leq i \leq 6$. So the loop $\alpha = (v_{1}, ..., v_{7})$ in $D$ also has length $7$.
We show that $\alpha$ is full. Suppose, by contradiction, that $v_{1} \sim v_{4}$. Then, due to Lemma \ref{3.1}, in $X$ we have $w_{1} \sim w_{4}$. Since $\beta$ is full, this implies a contradiction. So $v_{1} \nsim v_{4}$. One can similarly show that $v_{1} \nsim v_{3}$. Hence $\alpha$ is full.

Suppose by contradiction that $\alpha$ is not contained in the link of a vertex. Because $\alpha$ is full, there are at least two vertices in the interior of $\alpha$. Assume at first there are two such vertices, say $z$ and $y$. Obviously $z \sim y$. Assume w.l.o.g. $D_{z} = (y, v_{3}, v_{2}, v_{1}, v_{7})$ and $D_{y} = (v_{3}, v_{4}, v_{5}, v_{6}, v_{7}, z)$. For any other triangulation of $D$, we proceed similarly.
We consider a minimal filling diagram $(D',f')$ for $\gamma$ such that $D'$ is triangulated with the same simplices like $D$ except for the triangles $\langle z, v_{i}, v_{i+1} \rangle$, $1 \leq i \leq 2$, $\langle z, v_{3}, y \rangle$, $\langle z, v_{7}, y \rangle$, $\langle z, v_{7}, v_{1} \rangle$
in $D$ which are replaced in $D'$ by the triangles $\langle y, v_{1}, v_{7} \rangle$,
$\langle y, v_{i}, v_{i+1} \rangle, 1 \leq i \leq 2$. So in $D'$ the cycle $\alpha = (v_{1}, ..., v_{7})$ has a single interior vertex $y$. We define $f'$ such that it coincides with $f$ on all simplices which are common to $D$ and $D'$. We define $f'$ such that $f'(y) = x$,
$f'(\langle y,v_{i} \rangle) = \langle x,w_{i} \rangle$, $1 \leq i \leq 2$, $f'(\langle y, v_{i}, v_{i+1} \rangle) = \langle x,w_{i},w_{i+1} \rangle, 1 \leq i \leq 2$, $f'(\langle y, v_{1}, v_{7} \rangle) = \langle x,w_{1},w_{7} \rangle$. Since $f$ is simplicial, $f'$ is also simplicial. Hence $(D',f')$ is indeed a filling diagram for $\gamma$. Note that the area of $D'$ is less than the area of $D$. Because $D'$ has less interior vertices than $D$, this holds also due to Pick's formula.
 Based on the minimality of the area of $D$, we have reached a contradiction. Similarly, if there are at least three vertices inside $\alpha$, arguments similar to those above or Pick's formula, also imply a contradiction. Therefore $\alpha$ is contained in the link of a vertex. One can similarly show that any loop in $D$ of length less than $7$ but at least $4$, is also contained in the link of a vertex. Then $D$ is $7$-located.

%\begin{claim}\label{3.7}
%The loop $\beta$ is full in $X$.
%\end{claim}

%\begin{proof}
%FIG 60 PAG 60

%Suppose $\beta$ is not full. Suppose $w_{6} \sim w_{1}$ and $w_{6} \sim w_{2}$. Other cases can be treated similarly.

%According to the previous lemma, $f(y)$ differs from any vertex
%in $X_{f(y)}$. Moreover, the map $f$ is injective on $X_{f(y)}$.
%We construct a filling diagram $(D',f')$ for $\gamma$ such that in $D'$ we have $v_{6} \sim v_{1}$ and $v_{6} \sim v_{2}$. We triangulate $D'$ with the same triangles like $D$ except for the triangles $\langle v_{2},z,y %\rangle$, $\langle v_{6},z,y \rangle$, $\langle v_{0},v_{6},z \rangle$, $\langle v_{i},v_{i+1},z \rangle, 0 \leq i \leq 1$ in $D$ which are replaced in $D'$ by the triangles $\langle v_{2},v_{6},y \rangle$, $\langle %v_{6},v_{i},v_{i+1} \rangle, 0 \leq i \leq 1$. We define $f'$ such that it coincides with $f$ on all simplices except for $f'(\langle v_{6},v_{1} \rangle) = \langle w_{6},w_{1} \rangle$, $f'(\langle v_{6},v_{2} \rangle) = %\langle w_{6},w_{2} \rangle$, $f'(\langle v_{6},v_{i},v_{i+1} \rangle) = \langle w_{6},w_{i},w_{i+1} \rangle, 0 \leq i \leq 1$, $f'(\langle v_{2},v_{6},y \rangle) = \langle w_{2},w_{6},f(y) \rangle$.
% Then because $f$ is simplicial, $f'$ is also simplicial. $(D',f')$ is therefore indeed a filling diagram for $\gamma$. But $\rm{Area} D' = \rm{Area} D - 2$. Because $D$ has minimal area, this implies a contradiction. So $\beta$ %is full in $X$.

%\end{proof}

\end{proof}

The proof of the main result of the paper relies on the following lemma.
The proof follows closely, even up to the notations, the one given in \cite{ChaCHO} (Lemma $9.2$) for meshed graphs.

\begin{lemma}\label{3.8}
Let $X$ be a simply connected, $7$-located simplicial complex. Let $\gamma$ be a loop in $X$ and let $(D,f)$ be a minimal filling diagram for $\gamma$. Then for any three vertices $u,v,w$ of $D$ such that $v \sim w$ and for any shortest $(u,v)$-path $P$, there is a shortest $(u,w)$-path $Q$ such that $\rm{Area} C \leq \rm{const} \cdot d(u,v)$ where $C$ is the cycle formed by the paths $P, Q$ and the edge $\langle v,w \rangle$. We denote by $\rm{const}$ any natural number such that $\rm{const} > 2$.
\end{lemma}

\begin{proof}

Because $D$ is minimal, due to Lemma \ref{2.4}, the simplicial map $f$ is nondegenerate. Also the previous lemma implies that $D$ is $7$-located.

Let $k = d(u,v)$ and let $l = d(u,w)$. Let $v'$ be a vertex of $P$ such that $v \sim v'$. Let $P'$ be a shortest $(u,v')$-path such that $P = P' \cup \langle v',v \rangle$.
There are three cases to be analyzed: $l = k+1$, $l = k$, $l = k-1$.

\subsection{Case $1$}

We consider the case when $d(u,w) > d(u,v)$. Then $l = k+1$. Let $Q = P \cup \langle v,w \rangle$ be a shortest $(u,w)$-path. Then $\rm{Area} C = 0 \leq \rm{const} \cdot d(u,v)$. This completes the proof in this case.

\subsection{Case $2$}

We consider the case when $d(u,w) = d(u,v)$. Hence $l = k$. We prove by induction on $k$ the existence of a shortest $(u,w)$-path $Q$ such that $\rm{Area} C \leq \rm{const} \cdot k$ where $C$ is the cycle formed by the paths $P,Q$ and the edge $\langle v,w \rangle$.

We consider the case when $w \sim v'$. Also if $k = 1$, then $w \sim v'$. In both cases, let $Q = P' \cup \langle v',w \rangle$. Then $\rm{Area} C = 1 \leq \rm{const} \cdot k$ what completes the proof in these cases.

From now on we assume that $w \nsim v'$. We consider a $(v',w)$-path $(v', w_{1}, ..., w_{n}, w),$ $n \geq 1$ that does not pass through $v$ but, except for that, it is tightened. Let $\alpha = (v, v', w_{1}, ..., w_{n}, w, v)$. %Assume w.l.o.g. $w_{1} \notin P'$.

%ARAGRAFUL URMATOR E NECESAR SI LA CAZ 3?? E NECESAR LA CAZ 2??

%Assume first $n=1$. Note that $d(v',u) = d(w_{1},u) = k-1$. By induction hypothesis applied to $P'$, there is a shortest $(u,w_{1})$-path $Q'$ such that $\rm{Area} C_{1} \leq \rm{const} \cdot (k-1) + 2 \leq \rm{const} \cdot k$ %where $C_{1}$ is the cycle formed by the paths $P', Q'$ and the edge $\langle v', w_{1} \rangle$. Let $Q = Q' \cup \langle w_{1},w \rangle$. Because $D$ is simplicial, either $v' \sim w$ or $v \sim w_{1}$. Assume w.l.o.g. $w %\sim v'$. In conclusion $\rm{Area} C = \rm{Area} C_{1} + \rm{Area} (\langle v,v',w_{1} \rangle) + \rm{Area} (\langle v,w,w \rangle) \leq \rm{const} \cdot k$. For the rest of case $2$, let $n \geq 2$.

\begin{enumerate}

\item Case $2.1.$ The cycle $\alpha$ is full. Depending on the value of $n$, there are two cases to be analyzed. We treat them below.

%FIGURA 1 PAG 5
\begin{enumerate}

\item If $n \leq 4$, then $4 \leq |\alpha| \leq 7$. Because $\alpha$ is full and its length is at most $7$, by $7$-location, there is a vertex $z$ of $D$ such that $\alpha \subset D_{z}$. Note that $d(v',u) = d(z,u) = d(w_{n},u) = k-1$. Because $v' \sim z$, by induction on $P'$, there is a shortest $(u,z)$-path $Z$ such that $\rm{Area} C_{1} \leq \rm{const} \cdot (k-1)$ where $C_{1}$ is the cycle formed by the paths $P', Z$ and the edge $\langle v',z \rangle$. Because $z \sim w_{n}$, by induction on $Z$, there is a shortest $(u,w_{n})$-path $Q'$ such that $\rm{Area} C_{2} \leq \rm{const} \cdot (k-1)$ where $C_{2}$ is the cycle formed by the paths $Z,Q'$ and the edge $\langle z,w_{n} \rangle$. Let $Q = Q' \cup \langle w_{n},w \rangle$ be a shortest $(u,w)$-path. In conclusion $\rm{Area} C = \rm{Area} C_{1} + \rm{Area} C_{2} + \rm{Area}(\langle v,v',z \rangle) + \rm{Area}(\langle v,w,z \rangle) + \rm{Area}(\langle w,z,w_{n} \rangle) \leq 2 \cdot \rm{const} \cdot (k-1) + 3 \leq \rm{const} \cdot k$. The last inequality holds because $\rm{const} > 2$.

%FIG 7 PAG 7

\item If $n > 4$, then $|\alpha| > 7$.
Because $\alpha$ is full, there are vertices inside $\alpha$. Because the area of $D$ is minimal, based on Pick's formula, the number of vertices inside $\alpha$ is also minimal. Let $z_{j}, 1 \leq j \leq r$ be the vertices inside $\alpha$ such that $v' \sim z_{1}, z_{j} \sim z_{j+1}, 1 \leq j \leq r-1$, $z_{r} \sim w_{n}$. Besides, for $1 \leq j \leq r$, either $v \sim z_{j}$ or $w \sim z_{j}$ or $v \sim z_{j} \sim w$. Because $D$ is flat, there is a unique vertex $z_{q}$, $1 \leq q \leq r$ such that $v \sim z_{q} \sim w$. Note that $d(v',u) = d(z_{j},u) = k-1$, $1 \leq j \leq q$.  Because $v' \sim z_{1}$, by induction on $P'$, there is a shortest $(u,z_{1})$-path $Z_{1}$ such that $\rm{Area} C_{0} \leq \rm{const} \cdot (k-1)$ where $C_{0}$ is the cycle formed by $P', Z_{1}$, $\langle v',z_{1} \rangle$. For $j \in \{ 1, ..., q-1 \}$, because $z_{j} \sim z_{j+1}$, by induction on $Z_{j}$, there is a shortest $(u,z_{j+1})$-path $Z_{j+1}$ such that $\rm{Area} C_{j} \leq \rm{const} \cdot (k-1)$ where $C_{j}$ is the cycle formed by $Z_{j}, Z_{j+1}, \langle z_{j},z_{j+1} \rangle$. Let $Q = Z_{q} \cup \langle z_{q},w \rangle$ be a shortest $(u,w)$-path. Note that there are $q+1$ triangles contained in the cycle $(v,v',z_{1}, ..., z_{q},w)$. In conclusion $\rm{Area} C = \sum_{j=0}^{q-1} \rm{Area} C_{j} + \rm{Area} (\langle v,v',z_{1} \rangle) + ... + \rm{Area} (\langle v,z_{q},w \rangle) \leq q \cdot \rm{const} \cdot (k-1) + q + 1 \leq \rm{const} \cdot k$. The last inequality holds due to the fact that $\rm{const} > 2$.

\end{enumerate}

\item Case $2.2.$ The cycle $\alpha$ is not full.
Because $w \nsim v'$ and because the path $(v',w_{1}, ..., w_{n},w)$ is tightened (except for the fact that it does not pass through $v$), the possible diagonals of $\alpha$ are $\langle v,w_{i} \rangle, 1 \leq i \leq n$, $\langle w,w_{i} \rangle, 1 \leq i \leq n-1$. Depending on this, there are several cases to be analyzed. We treat them below.

Case $2.2.1.$ Suppose $v \sim w_{1}$. Note that $d(v',u) = d(w_{1},u) = k-1$. Because $v' \sim w_{1}$, by induction on $P'$, there is a shortest $(u,w_{1})$-path $R'$ such that $\rm{Area} C_{1} \leq \rm{const} \cdot (k-1)$. We denoted by $C_{1}$ the cycle formed by the paths $P',R', \langle v',w_{1} \rangle$. Then $\rm{Area} C' = \rm{Area} C_{1} + \rm{Area} (\langle v,v',w_{1} \rangle) \leq \rm{const} \cdot (k-1) + 1 \leq \rm{const} \cdot k$. We denoted by $C'$ the cycle formed by $P,R', \langle v,w_{1} \rangle$.

Case $2.2.2.$ Let $w_{i}, 2 \leq i \leq n$ such that $v \sim w_{i}$ and $v \nsim w_{i-j}, 1 \leq j \leq i-1$.
Let $\delta = (v,v',w_{1}, ..., w_{i})$. Note that, due to the choice of $w_{i}$, $\delta$ is full.
Depending on the value of $i$, there are two cases to be analyzed. We present them below.

\begin{enumerate}

 \item If $i \leq 5$ then $4 \leq |\delta| \leq 7$. Then, by $7$-location, there is a vertex $z$ such that $\delta \subset D_{z}$. Note that $d(v',u) = d(z,u) = d(w_{i},u) = k-1$. Because $v' \sim z$, by induction on $P'$, there is a shortest $(u,z)$-path $Z$ such that $\rm{Area} C_{1} \leq \rm{const} \cdot (k-1)$ where $C_{1}$ is the cycle formed by the paths $P', Z$ and the edge $\langle v',z \rangle$. Because $z \sim w_{i}$, by induction on $Z$, there is a shortest $(u,w_{i})$-path $R'$ such that $\rm{Area} C_{2} \leq \rm{const} \cdot (k-1)$ where $C_{2}$ is the cycle formed by the paths $Z, R'$ and the edge $\langle z,w_{i} \rangle$. Then we have $\rm{Area} C' = \rm{Area} C_{1} + \rm{Area} C_{2} + \rm{Area} (\langle v,v',z \rangle) + \rm{Area} (\langle v,w_{i},z \rangle$) $\leq 2 \cdot \rm{const} \cdot (k-1) + 2 \leq \rm{const} \cdot k$. We denoted by $C'$ the cycle formed by the paths $P,R'$ and the edge $\langle v,w_{i} \rangle$.

\item If $i > 5$ then $|\delta| > 7$. Because $\delta$ is full, there are vertices inside $\delta$. Because the area of $D$ is minimal, based on Pick's formula, the number of vertices inside $\delta$ is also minimal. Let $z_{j}, 1 \leq j \leq r$ be the vertices inside $\delta$ such that $v' \sim z_{1}, z_{j} \sim z_{j+1}, 1 \leq j \leq r-1, z_{r} \sim w_{i}$. Besides, for $1 \leq j \leq r$, $v \sim z_{j}$. Note that $d(v',u) = d(z_{j},u) = d(w_{i},u)$, $1 \leq j \leq r$. By induction on $P'$, there is a shortest $(u,z_{1})$-path $Z_{1}$ such that $\rm{Area} C_{0} \leq \rm{const} \cdot (k-1)$ where $C_{0}$ is the cycle formed by $P', Z_{1}$ and the edge $\langle v',z_{1} \rangle$. For $j \in \{ 1, ..., r-1 \}$, because $z_{j} \sim z_{j+1}$, by induction on $Z_{j}$, there is a shortest $(u,z_{j+1})$-path $Z_{j+1}$ such that $\rm{Area} C_{j} \leq \rm{const} \cdot (k-1)$ where $C_{j}$ is the cycle formed by $Z_{j}, Z_{j+1}, \langle z_{j},z_{j+1} \rangle$. Because $z_{r} \sim w_{i}$, by induction on $Z_{r}$, there is a shortest $(u,w_{i})$-path $R'$ such that $\rm{Area} C_{r} \leq \rm{const} \cdot (k-1)$ where $C_{r}$ is the cycle formed by $Z_{r}, R', \langle z_{r},w_{i} \rangle$. Note that there are $r+1$ triangles contained in the cycle $(v,v',z_{1}, ..., z_{r}, w_{i})$. In conclusion $\rm{Area} C' = \sum_{j=0}^{r} \rm{Area} C_{j} + \rm{Area} (\langle v,v',z_{1} \rangle) + ... + \rm{Area} (\langle v,w_{i},z_{r} \rangle) \leq (r+1) \cdot \rm{const} \cdot (k-1) + (r+1) \leq \rm{const} \cdot k$. We denoted by $C'$ the cycle formed by the paths $P, R'$ and the edge $\langle v,w_{i} \rangle$.

\end{enumerate}

Case $2.2.3.$ Let $w_{i}, 1 \leq i \leq n$ such that $w \sim w_{i}$, $w \nsim w_{i-j}, 1 \leq j \leq i-1$. We consider the cycle $\delta = (v,v',w_{1}, ..., w_{i}, w, v)$.

Case $2.2.3.a$ Assume $\delta$ is full. Depending on the value of $i$, there are two cases to be analyzed. We discuss them below.

\begin{enumerate}

\item If $i \leq 4$ then $4 \leq |\delta| \leq 7$. Then, by $7$-location, there is a vertex $z$ such that $\delta \subset D_{z}$. Note that $d(v',u) = d(z,u) = d(w_{i},u) = k-1$. Because $v' \sim z$, by induction on $P'$, there is a shortest $(u,z)$-path $Z$ such that $\rm{Area} C_{1} \leq \rm{const} \cdot (k-1)$ where $C_{1}$ is the cycle formed by the paths $P', Z$ and the edge $\langle v',z \rangle$. Because $z \sim w_{i}$, by induction on $Z$, there is a shortest $(u,w_{i})$-path $Q'$ such that $\rm{Area} C_{2} \leq \rm{const} \cdot (k-1)$ where $C_{2}$ is the cycle formed by the paths $Z, Q'$ and the edge $\langle z,w_{i} \rangle$. Let $Q = Q' \cup \langle w_{i},w \rangle$ be a shortest $(u,w)$-path. Then we have $\rm{Area} C = \rm{Area} C_{1} + \rm{Area} C_{2} + \rm{Area} (\langle v,v',z \rangle) + \rm{Area} (\langle v,w,z \rangle) + \rm{Area} (\langle w,w_{i},z \rangle)$ $\leq 2 \cdot \rm{const} \cdot (k-1) + 3 \leq \rm{const} \cdot k$.

\item If $i > 4$ then $|\delta| > 7$.
Because $\delta$ is full, there are vertices inside $\delta$. Because the area of $D$ is minimal, Pick's formula implies that the number of vertices inside $\delta$ is also minimal. Let $z_{j}, 1 \leq j \leq r$ be the vertices inside $\delta$ such that $v' \sim z_{1}, z_{j} \sim z_{j+1}, 1 \leq j \leq r-1, z_{r} \sim w_{i}$. Besides, for $1 \leq j \leq r$, either $v \sim z_{j}$ or $w \sim z_{j}$ or $v \sim z_{j} \sim w$. Because $D$ is flat, there is a unique vertex $z_{q}, 1 \leq q \leq r$ such that $v \sim z_{q} \sim w$. Note that $d(v',u) = d(z_{j},u) = k-1$, $1 \leq j \leq q$. Because $v' \sim z_{1}$, by induction on $P'$, there is a shortest $(u,z_{1})$-path $Z_{1}$ such that $\rm{Area} C_{0} \leq \rm{const} \cdot (k-1)$ where $C_{0}$ is the cycle formed by $P', Z_{1}$ and the edge $\langle v',z_{1} \rangle$. For $j \in \{ 1, ..., q-1 \}$, because $z_{j} \sim z_{j+1}$, by induction on $Z_{j}$, there is a shortest $(u,z_{j+1})$-path $Z_{j+1}$ such that $\rm{Area} C_{j} \leq \rm{const} \cdot (k-1)$ where $C_{j}$ is the cycle formed by $Z_{j}, Z_{j+1}, \langle z_{j},z_{j+1} \rangle$. Let $Q = Z_{q} \cup \langle z_{q},w \rangle$ be a shortest $(u,w)$-path. Note that there are $q+1$ triangles contained in the cycle $(v,v',z_{1}, ..., z_{q},w)$. In conclusion $\rm{Area} C = \sum_{j=0}^{q-1} \rm{Area} C_{j} + \rm{Area} (\langle v,v',z_{1} \rangle) + ...+ \rm{Area} (\langle v,z_{q},w \rangle) \leq q \cdot \rm{const} \cdot (k-1) + q+1 \leq \rm{const} \cdot k$.

\end{enumerate}

Case $2.2.3.b$ Assume $\delta$ is not full. 

Suppose $v \sim w_{1}$. Then $d(v',u) = d(w_{1},u) = k-1$. Because $v' \sim w_{1}$, by induction on $P'$, there is a shortest $(u,w_{1})$-path $R'$ such that $\rm{Area} C_{1} \leq \rm{const} \cdot (k-1)$ where $C_{1}$ is the cycle formed by $P', R', \langle v',w_{1} \rangle$. Then $\rm{Area} C' = \rm{Area} C_{1} + \rm{Area} (\langle v,v',w_{1} \rangle) \leq \rm{const} \cdot (k-1) + 1 \leq \rm{const} \cdot k$. We denoted by $C'$ the cycle formed by $P, R', \langle v,w_{1} \rangle$.

Let $w_{s}, 2 \leq s \leq i$ such that $v \sim w_{s}, v \nsim w_{s-j}, 1 \leq j \leq s-1$. Because the cycle $(v,v',w_{1}, ..., w_{s})$ is full and of length at least $4$, Case $2.2.2.$ implies that there is a shortest $(u,w_{s})$-path $R'$ such that $\rm{Area} C' \leq \rm{const} \cdot k$. We denoted by $C'$ the cycle formed by $P, R', \langle v,w_{s} \rangle$.

If $s = i$, let $Q = R' \cup \langle w_{i},w \rangle$ be a shortest $(u,w)$-path. Then $\rm{Area} C = \rm{Area} C' + \rm{Area} (\langle v,w,w_{i} \rangle) \leq \rm{const} \cdot (k-1) + 1 \leq \rm{const} \cdot k$.

From now on assume that $s \neq i$.
Let $\beta = (v,w_{s}, ..., w_{i}, w, v)$. In case $2.2.3.b.1$ we treat the situation when $\beta$ is full. In case $2.2.3.b.2$ we discuss the case when $\beta$ is not full.

Case $2.2.3.b.1$
Depending on the value of $i-s$, there are two cases to be analyzed. We present them below.

\begin{enumerate}

\item If $i-s \leq 4$, then $4 \leq |\beta| \leq 7$. By $7$-location, there is a vertex $z$ such that $\beta \subset D_{z}$. Note that $d(w_{s},u) = d(z,u) = d(w_{i},u) = k-1$. Because $w_{s} \sim z$, by induction on $R'$, there is a shortest $(u,z)$-path $Z$ such that $\rm{Area} C_{1} \leq \rm{const} \cdot (k-1)$ where $C_{1}$ is the cycle formed by the paths $R', Z$ and the edge $\langle w_{s},z \rangle$. Because $z \sim w_{i}$, by induction on $Z$, there is a shortest $(u,w_{i})$-path $Q'$ such that $\rm{Area} C_{2} \leq \rm{const} \cdot (k-1)$ where $C_{2}$ is the cycle formed by the paths $Z, Q'$ and the edge $\langle z,w_{i} \rangle$. Let $Q = Q' \cup \langle w_{i},w \rangle$ be a shortest $(u,w)$-path. Then we have $\rm{Area} C'' = \rm{Area} C_{1} + \rm{Area} C_{2} + \rm{Area} (\langle w_{s},v,z \rangle) + \rm{Area} (\langle v,w,z \rangle) + \rm{Area} (\langle w,w_{i},z \rangle)$ $\leq 2 \cdot \rm{const} \cdot (k-1) + 3 \leq \rm{const} \cdot k$. We denoted by $C''$ the cycle formed by the paths $R',Q$ and the edges $\langle w_{s},v \rangle$, $\langle v,w \rangle$.

   In conclusion $\rm{Area} C  = \rm{Area} C' + \rm{Area} C'' \leq  \rm{const} \cdot k$.

\item If $i-s > 4$, then $|\beta| > 7$.
Because $\beta$ is full, there are vertices inside $\beta$. Because the area of $D$ is minimal, based on Pick's formula, the number of vertices inside $\beta$ is also minimal. Let $z_{j}, 1 \leq j \leq r$ be the vertices inside $\beta$ such that $w_{s} \sim z_{1}, z_{j} \sim z_{j+1}, 1 \leq j \leq r-1, z_{r} \sim w_{i}$. Besides, for $1 \leq j \leq r$, either $v \sim z_{j}$ or $w \sim z_{j}$ or $v \sim z_{j} \sim w$. Because $D$ is flat, there is a unique vertex $z_{q}, 1 \leq q \leq r$ such that $v \sim z_{q} \sim w$. Note that $d(w_{s},u) = d(z_{j},u) = k-1$, $1 \leq j \leq q$. By induction on $R'$, because $w_{s} \sim z_{1}$, there is a shortest $(z_{1},u)$-path $Z_{1}$ such that $\rm{Area} C_{0} \leq \rm{const} \cdot (k-1)$ where $C_{0}$ is the cycle formed by $R', Z_{1}$ and the edge $\langle w_{s},z_{1} \rangle$. For $j \in \{ 1, ..., q-1 \}$, by induction on $Z_{j}$, because $z_{j} \sim z_{j+1}$, there is a shortest $(u,z_{j+1})$-path $Z_{j+1}$ such that $\rm{Area} C_{j} \leq \rm{const} \cdot (k-1)$ where $C_{j}$ is the cycle formed by $Z_{j}, Z_{j+1}, \langle z_{j},z_{j+1} \rangle$. Let $Q = Z_{q} \cup \langle z_{q},w \rangle$ be a shortest $(u,w)$-path. Note that there are $q+1$ triangles contained in the cycle $(v,w_{s},z_{1}, ..., z_{q},w)$. In conclusion $\rm{Area} C'' = \sum_{j=0}^{q-1} \rm{Area} C_{j} + \rm{Area} (\langle v,w_{s},z_{1} \rangle) + ... + \rm{Area} (\langle v,z_{q},w \rangle) \leq q \cdot \rm{const} \cdot (k-1) + q+1  \leq \rm{const} \cdot k$. We denoted by $C''$ the cycle formed by the paths $R',Q$ and the edges $\langle w_{s},v \rangle$, $\langle v,w \rangle$.

   In conclusion $\rm{Area} C  = \rm{Area} C' + \rm{Area} C'' \leq  \rm{const} \cdot k$.

Case $2.2.3.b.2$
If $\beta$ is not full, we split $\beta$ into full cycles $\beta_{1}, ..., , \beta_{j}, j \geq 2$  containing each at least one edge of $\beta$ and at least one of its diagonals. We argue for each of these full cycles the same way we argued in one of the cases discussed above.

\end{enumerate}

\end{enumerate}

\subsection{Case $3$}

We consider the case when $d(u,w) < d(u,v)$. Hence $l = k-1$. We prove by induction on $k$ the existence of a shortest $(u,w)$-path $Q$ such that $\rm{Area} C \leq \rm{const} \cdot k$ where $C$ is the cycle formed by the paths $P,Q$ and the edge $\langle v,w \rangle$.

If $k = 1$ or more generally if $w = v'$, then let $Q = P'$ be a shortest $(u,w)$-path. In this case we get $\rm{Area} C = 0 \leq \rm{const} \cdot k$.

If $v' \sim w$, note that $d(v',u) = d(w,u) = k-1$. By induction on $P'$, there is a shortest $(u,w)$-path $Q$ such that $\rm{Area} C' \leq  \rm{const} \cdot (k-1)$ where $C'$ is the cycle formed by the paths $P',Q$ and the edge $\langle v',w \rangle$. Then we have $\rm{Area} C = \rm{Area} C' + \rm{Area} (\langle v,v',w \rangle) \leq \rm{const} \cdot (k-1) + 1 \leq \rm{const} \cdot k$.

%FIG 8 PAG 7

From now on assume that $w \neq v'$ and that $w \nsim v'$.

We consider a $(v',w)$-path $(v', w_{1}, ..., w_{n}, w)$ that does not pass through $v$ but, except for that, it is tightened. Let $\alpha = (v,v', w_{1}, ..., w_{n}, w,v), n \geq 1$.

\begin{enumerate}

\item Case $3.1.$ The cycle $\alpha$ is full. Depending on the value of $n$, there are two cases to be analyzed.

%FIG 9 PAG 7
\begin{enumerate}

\item If $n \leq 4$, then $4 \leq |\alpha| \leq 7$.
By $7$-location, there is a vertex $z$ such that $\alpha \subset D_{z}$.
Note that $d(v',u) = d(z,u) = d(w,u) = k-1$.
Because $v' \sim z$, by induction on $P'$, there is a shortest $(u,z)$-path $Z$ such that $\rm{Area} C_{1} \leq \rm{const} \cdot (k-1)$ where $C_{1}$ is the cycle formed by $P',Z$, $\langle v',z \rangle$.
Because $z \sim w$, by induction on $Z$, there is a shortest $(u,w)$-path $Q$ such that $\rm{Area} C_{2} \leq \rm{const} \cdot (k-1)$ where $C_{2}$ is the cycle formed by $Z,Q$, $\langle z,w \rangle$.
In conclusion $\rm{Area} C = \rm{Area} C_{1} + \rm{Area} C_{2} + \rm{Area} (\langle v',v,z \rangle) + \rm{Area} (\langle w,v,z \rangle) \leq 2 \cdot \rm{const} \cdot (k-1) + 2 \leq \rm{const} \cdot k$.

%FIG 10 PAG 7

\item If $n > 4$, then $|\alpha| > 7$.
Because $\alpha$ is full, there are vertices inside $\alpha$. Because the area of $D$ is minimal, due to Pick's formula, the number of vertices inside $\alpha$ is also minimal. Let $z_{j}, 1 \leq j \leq r$ be the vertices inside $\alpha$ such that $v' \sim z_{1}$, $z_{j} \sim z_{j+1}, 1 \leq j \leq r-1$, $z_{r} \sim w_{n}$. Besides, for $1 \leq j \leq r$, either $v \sim z_{j}$ or $w \sim z_{j}$ or $v \sim z_{j} \sim w$. Because $D$ is flat, there is a unique vertex $z_{q}, 1 \leq q \leq r$ such that $v \sim z_{q} \sim w$. Note that $d(v',u) = d(z_{j},u) = d(w,u) = k-1$, $1 \leq j \leq q$. Because $v' \sim z_{1}$, by induction on $P'$, there is a shortest $(u,z_{1})$-path $Z_{1}$ such that $\rm{Area} C_{0} \leq \rm{const} \cdot (k-1)$ where $C_{0}$ is the cycle formed by $P', Z_{1}$ and the edge $\langle v',z_{1} \rangle$. For $j \in \{ 1, ..., q-1 \}$, because $z_{j} \sim z_{j+1}$, by induction on $Z_{j}$, there is a shortest $(u,z_{j+1})$-path $Z_{j+1}$ such that $\rm{Area} C_{j} \leq \rm{const} \cdot (k-1)$ where $C_{j}$ is the cycle formed by $Z_{j}, Z_{j+1}, \langle z_{j},z_{j+1} \rangle$. Because $z_{q} \sim w$, by induction on $Z_{q}$, there is a shortest $(u,w)$-path $Q$ such that $\rm{Area} C_{q} \leq \rm{const} \cdot (k-1)$ where $C_{q}$ is the cycle formed by $Z_{q}, Q, \langle z_{q},w \rangle$. Note that there are $q+1$ triangles contained in the cycle $(v,v',z_{1}, ..., z_{q},w)$. In conclusion $\rm{Area} C = \sum_{j=0}^{q} \rm{Area} C_{j} + \rm{Area} (\langle v,v',z_{1} \rangle) + ... + \rm{Area} (\langle v,z_{q},w \rangle) \leq q \cdot \rm{const} \cdot (k-1) + q+1 \leq \rm{const} \cdot k$.

\end{enumerate}

\item Case $3.2.$ The cycle $\alpha$ is not full.
Because $w \nsim v'$ and because the path $(v',w_{1}, ..., w_{n},w)$ is tightened (except for the fact that it does not pass through $v$), the possible diagonals of $\alpha$ are $\langle v,w_{i} \rangle, 1 \leq i \leq n$, $\langle w,w_{i} \rangle, 1 \leq i \leq n-1$. Depending on this, there are several cases to be analyzed. We treat them below.

Case $3.2.1.$ Suppose $v \sim w_{1}$. Case $2.2.1.$ implies that there is a shortest $(u,w_{1})$-path $R'$ such that $\rm{Area} C' \leq \rm{const} \cdot k$. We denoted by $C'$ the cycle formed by $P,R', \langle v,w_{1} \rangle$.

%Note that $d(v',u) = d(w_{1},u) = k-1$. Because $v' \sim w_{1}$, by induction on $P'$, there is a shortest $(u,w_{1})$-path $R'$ such that $\rm{Area} C_{1} \leq \rm{const} \cdot (k-1)$ where $C_{1}$ is the cycle formed by the %paths $P',R', \langle v',w_{1} \rangle$. Then $\rm{Area} C' = \rm{Area} C_{1} + \rm{Area} (\langle v,v',w_{1} \rangle) \leq \rm{const} \cdot (k-1) + 1 \leq \rm{const} \cdot k$. The last inequality holds due to the fact that %$\rm{const} > 2$. We denoted by $C'$ the cycle formed by $P,R', \langle v,w_{1} \rangle$.

Case $3.2.2.$ Let $w_{i}, 2 \leq i \leq n$ such that $v \sim w_{i}$ and $v \nsim w_{i-j}, 1 \leq j \leq i-1$.
Let $\delta = (v,v',w_{1}, ..., w_{i})$. Due to the choice of $w_{i}$, $\delta$ is full. Case $2.2.2.$ implies that there is a shortest $(u,w_{i})$-path $R'$ such that $\rm{Area} C' \leq \rm{const} \cdot k$. We denoted by $C'$ the cycle formed by $P, R'$, $\langle v,w_{i} \rangle$.

Case $3.2.3.$ Let $w_{i}, 1 \leq i \leq n$ such that $w \sim w_{i}$, $w \nsim w_{i-j}, 1 \leq j \leq i-1$. We consider the cycle $\delta = (v,v',w_{1}, ..., w_{i}, w, v)$.

Case $3.2.3.a$ Assume $\delta$ is full.
Depending on the value of $i$, there are two cases to be analyzed. We present them below.

\begin{enumerate}

\item If $i \leq 4$ then $4 \leq |\delta| \leq 7$. By $7$-location, there is a vertex $z$ such that $\delta \subset D_{z}$. Note that $d(v',u) = d(z,u) = d(w,u) = k-1$. Because $v' \sim z$, by induction on $P'$, there is a shortest $(u,z)$-path $Z$ such that $\rm{Area} C_{1} \leq \rm{const} \cdot (k-1)$ where $C_{1}$ is the cycle formed by the paths $P', Z$ and the edge $\langle v',z \rangle$. Because $z \sim w$, by induction on $Z$, there is a shortest $(u,w)$-path $Q$ such that $\rm{Area} C_{2} \leq \rm{const} \cdot (k-1)$ where $C_{2}$ is the cycle formed by the paths $Z, Q$ and the edge $\langle z,w \rangle$. Then we have $\rm{Area} C = \rm{Area} C_{1} + \rm{Area} C_{2} + \rm{Area} (\langle v,v',z \rangle) + \rm{Area} (\langle v,w,z \rangle) $ $\leq 2 \cdot \rm{const} \cdot (k-1) + 2 \leq \rm{const} \cdot k$.

\item If $i > 4$ then $|\delta| > 7$.
Because $\delta$ is full, there are vertices inside $\delta$. Because the area of $D$ is minimal, due to Pick's formula, the number of vertices inside $\delta$ is also minimal. Let $z_{j}, 1 \leq j \leq r$ be the vertices inside $\delta$ such that $v' \sim z_{1}, z_{j} \sim z_{j+1}, 1 \leq j \leq r-1, z_{r} \sim w$. Besides, for $1 \leq j \leq r$, either $v \sim z_{j}$ or $v \sim z_{j} \sim w$ or $w \sim z_{j}$. Because $D$ is flat, there is a unique vertex $z_{q}$ such that $v \sim z_{q} \sim w$, $1 \leq q \leq r$. Note that $d(v',u) = d(z_{j},u) = d(w,u) = k-1$, $1 \leq j \leq r$. By induction on $P'$, because $v' \sim z_{1}$, there is a shortest $(u,z_{1})$-path $Z_{1}$ such that $\rm{Area} C_{0} \leq \rm{const} \cdot (k-1)$ where $C_{0}$ is the cycle formed by $P', Z_{1}$, $\langle v',z_{1} \rangle$. For $j \in \{ 1, ..., q-1 \}$, by induction on $Z_{j}$, because $z_{j} \sim z_{j+1}$, there is a shortest $(u,z_{j+1})$-path $Z_{j+1}$ such that $\rm{Area} C_{j} \leq \rm{const} \cdot (k-1)$ where $C_{j}$ is the cycle formed by $Z_{j}, Z_{j+1}, \langle z_{j},z_{j+1} \rangle$. Because $z_{q} \sim w$, by induction on $Z_{q}$, there is a shortest $(u,w)$-path $Q$ such that $\rm{Area} C_{q} \leq \rm{const} \cdot (k-1)$ where $C_{q}$ is the cycle formed by $Z_{q}, Q, \langle z_{q},w \rangle$. Note that there are $q+1$ triangles contained in the cycle $(v, v', z_{1}, ..., z_{q},w)$. In conclusion $\rm{Area} C = \sum_{j=0}^{q} \rm{Area} C_{j} + \rm{Area} (\langle v,v',z_{1} \rangle) + ... + \rm{Area} (\langle v,w,z_{q} \rangle) \leq (q+1) \cdot \rm{const} \cdot (k-1) + q+1 \leq \rm{const} \cdot k$.

\end{enumerate}

Case $3.2.3.b$ Assume $\delta$ is not full. 

Suppose $v \sim w_{1}$. Then $d(v',u) = d(w_{1},u) = k-1$. Because $v' \sim w_{1}$, by induction on $P'$, there is a shortest $(u,w_{1})$-path $R'$ such that $\rm{Area} C_{1} \leq \rm{const} \cdot (k-1)$ where $C_{1}$ is the cycle formed by $P', R', \langle v',w_{1} \rangle$. Then $\rm{Area} C' = \rm{Area} C_{1} + \rm{Area} (\langle v,v',w_{1} \rangle) \leq \rm{const} \cdot (k-1) + 1 \leq \rm{const} \cdot k$. We denoted by $C'$ the cycle formed by $P, R', \langle v,w_{1} \rangle$.

Let $w_{s}, 2 \leq s \leq i$ such that $v \sim w_{s}, v \nsim w_{s-j}, 1 \leq j \leq s-1$. Because $(v,v',w_{1}, ..., w_{s})$ is full and of length at least $4$, Case $2.2.2.$ implies that there is a shortest $(u,w_{s})$-path $R'$ such that $\rm{Area} C' \leq \rm{const} \cdot (k-1)$. We denoted by $C'$ the cycle formed by $P,R', \langle v,w_{s} \rangle$.

If $s = i$, let $Q = R' \cup \langle w_{i},w \rangle$ be a shortest $(u,w)$-path. Then $\rm{Area} C = \rm{Area} C' + \rm{Area} (\langle v,w,w_{i} \rangle) \leq \rm{const} \cdot (k-1) + 1 \leq \rm{const} \cdot k$.

From now on assume that $s \neq i$.
Let $\beta = (v,w_{s}, ..., w_{i}, w, v)$. In case $3.2.3.b.1$ we treat the situation when $\beta$ is full. In case $3.2.3.b.2$ we consider the case when $\beta$ is not full.

Case $3.2.3.b.1$
Depending on the value of $i-s$, there are two cases to be analyzed. We present them below.

\begin{enumerate}

\item If $i-s \leq 4$ then $4 \leq |\beta| \leq 7$. By $7$-location, there is a vertex $z$ such that $\beta \subset D_{z}$. Note that $d(w_{s},u) = d(z,u) = d(w,u) = k-1$. Because $w_{s} \sim z$, by induction on $R'$, there is a shortest $(u,z)$-path $Z$ such that $\rm{Area} C_{1} \leq \rm{const} \cdot (k-1)$ where $C_{1}$ is the cycle formed by $R', Z$, $\langle w_{s},z \rangle$. Because $z \sim w$, by induction on $Z$, there is a shortest $(u,w)$-path $Q$ such that $\rm{Area} C_{2} \leq \rm{const} \cdot (k-1)$ where $C_{2}$ is the cycle formed by the paths $Z, Q$ and the edge $\langle z,w \rangle$. Then we have $\rm{Area} C'' = \rm{Area} C_{1} + \rm{Area} C_{2} + \rm{Area} (\langle w_{s},v,z \rangle) + \rm{Area} (\langle v,w,z \rangle) $ $\leq 2 \cdot \rm{const} \cdot (k-1) + 2 \leq \rm{const} \cdot k$. We denoted by $C''$ the cycle formed by the paths $R',Q$ and the edges $\langle w_{s},v \rangle$, $\langle v,w \rangle$.

   In conclusion $\rm{Area} C  = \rm{Area} C' + \rm{Area} C'' \leq  \rm{const} \cdot k$.

\item If $i-s > 4$ then $|\beta| > 7$.
Because $\beta$ is full, there are vertices inside $\beta$. Because the area of $D$ is minimal, based on Pick's formula, the number of vertices inside $\beta$ is also minimal. Let $z_{j}, 1 \leq j \leq r$ be the vertices inside $\beta$ such that $w_{s} \sim z_{1}, z_{j} \sim z_{j+1}, 1 \leq j \leq r, z_{r} \sim w$. Besides, $v \sim z_{j}$, $1 \leq j \leq r-1$, $v \sim z_{r} \sim w$. Note that $d(w_{s},u) = d(z_{j},u) = d(w,u) = k-1$, $1 \leq j \leq r$. Because $w_{s} \sim z_{1}$, by induction on $R'$, there is a shortest $(u,z_{1})$-path $Z_{1}$ such that $\rm{Area} C_{0} \leq \rm{const} \cdot (k-1)$ where $C_{0}$ is the cycle formed by $R', Z_{1}$, $\langle w_{s},z_{1} \rangle$. For $j \in \{ 1, ..., r-1\}$, because $z_{j} \sim z_{j+1}$, by induction on $Z_{j}$, there is a shortest $(u,z_{j+1})$-path $Z_{j+1}$ such that $\rm{Area} C_{j} \leq \rm{const} \cdot (k-1)$ where $C_{j}$ is the cycle formed by $Z_{j}, Z_{j+1}, \langle z_{j},z_{j+1} \rangle$. Because $z_{r} \sim w$, by induction on $Z_{r}$, there is a shortest $(u,w)$-path $Q$ such that $\rm{Area} C_{r} \leq \rm{const} \cdot (k-1)$ where $C_{r}$ is the cycle formed by $Z_{r}, Q, \langle z_{r},w \rangle$. Note that there are $r+1$ triangles contained in the cycle $(v, w_{s}, z_{1}, ..., z_{r}, w)$. In conclusion $\rm{Area} C'' = \sum_{j=0}^{r} \rm{Area} C_{j} + \rm{Area} (\langle v,w_{s},z_{1} \rangle) + ... + \rm{Area} (\langle v,z_{r},w \rangle) \leq (r+1) \cdot \rm{const} \cdot (k-1) + r+1 \leq \rm{const} \cdot k$. We denoted by $C''$ the cycle formed by the paths $R',Q$ and $\langle w_{s},v \rangle$, $\langle v,w \rangle$.

   In conclusion $\rm{Area} C  = \rm{Area} C' + \rm{Area} C'' \leq  \rm{const} \cdot k$.

Case $3.2.3.b.2$
If $\beta$ is not full, we split $\beta$ into full cycles $\beta_{1}, ..., , \beta_{j}, j \geq 2$ containing each at least one edge of $\beta$ and at least one of its diagonals. We argue for each of these full cycles the same way we argued in one of the cases discussed above.

\end{enumerate}

\end{enumerate}

\end{proof}

\begin{theorem}[quadratic isoperimetric inequality]
Let $X$ be a simply connected, $7$-located simplicial complex. Let $\gamma$ be a loop in $X$. Let $(D,f)$ be a minimal filling diagram for $\gamma$. Let $D$ be triangulated such that it contains a cycle $\alpha = (v_{0}, v_{1}, ..., v_{n-1}, v_{0})$ of length $n$. Let $D_{0}$ be the subdisc of $D$ bounded by $\alpha$. Then $\rm{Area} (D_{0}) < \rm{const} \cdot n^{2}$ where $\rm{const}$ is any natural number such that $\rm{const} > 2$.
\end{theorem}

\begin{proof}

 For each $i \in \{0, ..., n-1\}$, we define a shortest $(v_{0},v_{i})$-path $P_{i}$ such that $\rm{Area} C_{i} \leq \rm{const} \cdot d(v_{0}, v_{i})$ where $C_{i}$ is the cycle formed by the paths $P_{i}$, $P_{i+1}$ and the edge $\langle v_{i}, v_{i+1} \rangle$. Let $P_{0}$ be the one vertex path $v_{0}$. Assume that we have already constructed $P_{i}$ such that it intersects $P_{i-1}$ only in $v_{0}$. According to the previous lemma, there exists a path $P_{i+1}$ in $D_{0}$ from $v_{0}$ to $v_{i+1}$ that intersects $P_{i}$ only in $v_{0}$ such that $\rm{Area} C_{i} \leq \rm{const} \cdot d(v_{0},v_{i}) \leq \rm{const} \cdot \frac{n}{2} < 2 \cdot \rm{const} \cdot \frac{n}{2} = \rm{const} \cdot n, 0 \leq i \leq n-1$. We denote by $C_{i}$ the concatenation of the paths $P_{i}, P_{i+1}$ and the edge $\langle v_{i},v_{i+1} \rangle$. Hence one can fill $D_{0}$ using the collection of $n$ cycles $C_{0}, C_{1}, ..., C_{n-1}$ each satisfying the inequality $\rm{Area} C_{i} < \rm{const} \cdot n, 0 \leq i \leq n-1$. In conclusion we have $\rm{Area} D_{0} \leq n \cdot \rm{Area} C_{i} < \rm{const} \cdot n^{2}$, $0 \leq i \leq n-1$.
\end{proof}

\end{document}